\documentclass[reqno,11pt]{amsart}
\usepackage{amssymb,amsmath,euler}
\usepackage{tikz,enumerate}
\usetikzlibrary{patterns}

\numberwithin{equation}{section}
\newtheorem{theorem}[equation]{Theorem}

\newtheorem{proposition}[equation]{Proposition}

\theoremstyle{definition}

\newtheorem*{remark}{Remark}

\def\A{\mathscr{A}}
\def\B{\mathscr{B}}
\def\L{\mathscr{L}}
\def\M{\mathscr{M}}
\def\P{\mathscr{P}}
\def\T{\mathscr{T}}

\def\U{\mathsf{U}}
\def\D{\mathsf{D}}
\def\H{\mathsf{H}}
\def\PM{\mathsf{M}}

\tikzstyle{arc1}=[draw, line width=1.2]
\tikzstyle{patt1}=[pattern=north west lines, pattern color=gray!60, thick]

\newcommand{\drawmatching}[3]{%
\def\sf{#1}
\def\n{#2}
\tikz[scale=\sf]{%
\draw[gray!50] (1,0)--(\n,0);
\foreach \i in {1,...,\n} {%
	\draw[arc1,fill] (\i,0) circle (0.06);
}
\foreach \x/\y in {#3} {
	\pgfmathsetmacro\s{\x+\y}
	\pgfmathsetmacro\d{\y-\x}
	\draw[arc1] (\x,0) parabola bend (0.5*\s,0.2*\d) (\y,0);
}
}}

\newcommand{\drawpath}[4]{%
\def\sf{#1}
\def\n{#2}
\def\max{#3}
\def\y{{#4}}
\tikz[scale=\sf, baseline=\max]{%
\draw[thin,gray!50] (1,0) grid (\n+1,\max);
\draw[fill] (\n+1,\y[\n]) circle (0.06);
\foreach \i in {1,...,\n} {%
	\draw[thick] (\i,\y[\i-1])--(\i+1,\y[\i]);
	\draw[fill] (\i,\y[\i-1]) circle (0.06);
}
}}

\tikzstyle{path}=[draw, line width=1.2, color=black]

\title{Enumeration of symmetric arc diagrams}
\author{Juan B. Gil}
\address{Penn State Altoona \\ Altoona, PA 16601}
\email{jgil@psu.edu}

\author{Luis E. Lopez}
\email{lql9872@protonmail.com}

\begin{document}
\begin{abstract}
We give recurrence relations for the enumeration of symmetric elements within four classes of arc diagrams corresponding to certain involutions and set partitions whose blocks contain no consecutive integers. These arc diagrams are motivated by the study of RNA secondary structures. For example, classic RNA secondary structures correspond to 3412-avoiding involutions with no adjacent transpositions, and structures with base triples may be represented as partitions with crossings. Our results rely on combinatorial arguments. In particular, we use Motzkin paths to describe noncrossing arc diagrams that have no arc connecting two adjacent nodes, and we give an explicit bijection to ternary words whose length coincides with the sum of their digits. We also discuss the asymptotic behavior of some of the sequences considered here in order to quantify the extremely low probability of finding symmetric structures with a large number of nodes.
\end{abstract}

\maketitle

\section{Introduction}

In its primary form, RNA is a single-stranded molecule composed of nucleotides with a nitrogenous base. Through hydrogen bonds between their bases, nonadjacent nucleotides interact to form base pairs/triples that are responsible for the folding of the RNA sequence into a helical configuration. The planar representation of such a configuration is known as secondary structure of the RNA molecule.

The first graph theoretical definition of a secondary structure was given in the late 70's by Michael Waterman \cite{Waterman}. Ever since, secondary structures have become an important subject of mathematical research, including their enumeration, asymptotics, pattern avoidance, and their connection to other combinatorial objects, see e.g.\ \cite{DSV04,HaSt99,HSS98,MuNe15,JQR08,SchWat94,Willenbring}. 

The simplest RNA secondary structures can be represented by noncrossing arc diagrams, where the nucleotides of the single-stranded primary structure are described by nodes on a line segment, and base pairs are represented by noncrossing arcs between nodes that form a bond (see Figure~\ref{fig:simpleEx}). Such structures were studied in \cite{Waterman} under the aspects of enumeration, classification by complexity, and algorithms aimed at finding the most stable\footnote{Having minimum free energy.} secondary structures.

\begin{figure}[ht]
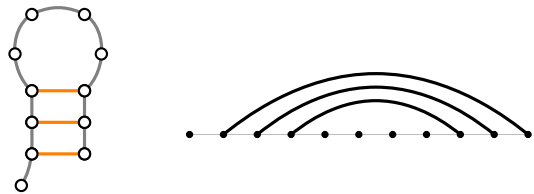

\tikzstyle{n}=[circle, inner sep=1.5, fill=white, draw=black, thick]
\tikzstyle{curve}=[arc1,gray]
\tikz[scale=0.7,baseline=6pt]{
 \draw[curve] (-0.2,-0.6) to[out=60,in=270] (0,0.6) to[out=90,in=270] (0,1.2);
 \draw[curve] (0,1.2) to[out=140,in=200] (0.1,2.7) to[out=20,in=160](0.9,2.7) to[out=-20,in=40](1,1.2);
 \draw[curve] (1,1.2) to[out=270,in=90] (1,0);
 \foreach \y in {0,1,2} {\node[n] at (1,0.6*\y) {};}
 \foreach \y in {0,1,2} {\node[n] at (0,0.6*\y) {};}
 \foreach \y in {0,1,2} {\draw[arc1,orange] (0,0.6*\y) -- (1,0.6*\y);} 
 \foreach \y in {0,1,2} {\node[n] at (1,0.6*\y) {};}
 \foreach \y in {0,1,2} {\node[n] at (0,0.6*\y) {};}
 \foreach \x/\y in {-0.32/1.9,0/2.65,1/2.65,1.32/1.9} {\node[n] at (\x,\y) {};}
 \foreach \x/\y in {-0.2/-0.6} {\node[n] at (\x,\y) {};}
}
\hspace{3ex}
\drawmatching{0.45}{11}{2/11,3/10,4/9}
\caption{RNA secondary structure and its arc diagram representation.}
\label{fig:simpleEx}
\end{figure}

Another important class of RNA structures are the so-called pseudoknots. They form when nucleotides from a single-stranded region inside a hairpin loop bond with nucleotides from another region of the molecule (see Figure~\ref{fig:pseudoknotEx}). Pseudoknots give rise to arc diagrams with crossings. They correspond to self-inverse permutations, also known as {\em involutions}, with no transpositions of the form $(i,i+1)$. 

\begin{figure}[ht]
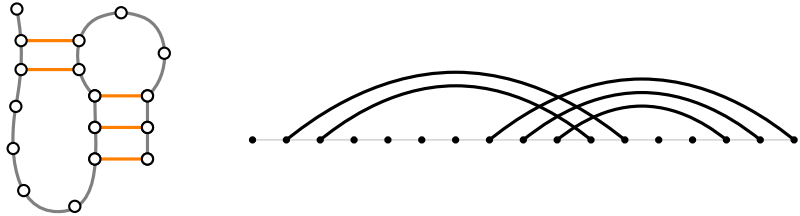

\tikzstyle{n}=[circle, inner sep=1.5, fill=white, draw=black, thick]
\tikzstyle{curve}=[gray, arc1]
\tikz[scale=0.7,baseline=6pt]{
 \draw[curve] (-1.48,2.85) to[out=-80,in=80] (-1.5,1) to[out=260,in=180] (-0.7,-1);
 \draw[curve] (-0.7,-1) to[out=0,in=270] (0,1.2);
 \draw[curve] (0,1.2) to[out=140,in=200] (0.1,2.7) to[out=20,in=160](0.9,2.7) to[out=-20,in=40](1,1.2);
 \draw[curve] (1,1.2) to[out=270,in=90] (1,0);
 \foreach \y in {3.4,4.5} {
 \draw[arc1,orange] (-1.4,0.5*\y) -- (-0.3,0.5*\y);
 	\node[n] at (-1.4,0.5*\y) {};
	\node[n] at (-0.3,0.5*\y) {};} 
 \foreach \y in {0,1,2} {\node[n] at (1,0.6*\y) {};}
 \foreach \y in {0,1,2} {\node[n] at (0,0.6*\y) {};}
 \foreach \y in {0,1,2} {\draw[very thick,orange] (0,0.6*\y) -- (1,0.6*\y);} 
 \foreach \y in {0,1,2} {\node[n] at (1,0.6*\y) {};}
 \foreach \y in {0,1,2} {\node[n] at (0,0.6*\y) {};}
 \node[n] at (-1.48,2.85) {};
 \foreach \x/\y in {-1.5/1,-1.55/0.2,-1.35/-0.6,-0.38/-0.9} {\node[n] at (\x,\y) {};}
 \foreach \x/\y in {0.5/2.78,1.32/2.01} {\node[n] at (\x,\y) {};}
}
\hspace{3ex}
\drawmatching{0.45}{17}{2/12,3/11,8/17,9/16,10/15}
\caption{RNA structure with pseudoknots.}
\label{fig:pseudoknotEx}
\end{figure}

Much is known about the enumeration of structures with pseudoknots (see e.g.\ \cite{HaSt99,JQR08}), but we did not find formulas for the enumeration of symmetric elements.

More recently, there has been increasing interest in developing computational models that include base triples. A base triple forms when a nucleobase interacts with two other bases by edge-to-edge hydrogen bonding, and as described in Abu Almakarem et al.~\cite{APSZL12}, these interactions are actually common in structured RNA molecules. In 2012, H\"oner zu Siederdissen et al.~\cite{HBSH11} proposed an extended model together with a folding algorithm that allows the prediction of RNA secondary structures with base triples. Based on that model, M\"uller and Nebel \cite{MuNe15} developed a combinatorial framework for the study of such extended structures (excluding pseudoknots), and provided enumerative results for several statistics similar to the ones usually considered for classic secondary structures.

In this paper, we address the enumeration of symmetric elements within the aforementioned families of secondary structures, including structures with pseudoknots and certain types of base triples. While global symmetry is rather unlikely in actual RNA secondary structures, most RNA molecules exhibit local symmetry or are predominantly symmetric (notably riboswitches). As discussed by Waterman \cite{Waterman}, symmetric structures are likely to arise when solving for the most stable secondary structures of a given RNA sequence. For example, for the sequence
\begin{center}
 {\bf s} $=$ {\sc ggguaunnnauagggnnncccauannnuauccc}, 
\end{center}
the best first and second order secondary structures are symmetric. For the definition of {\em $k$th order} secondary structures, we refer the reader to \cite[Definition~2.3]{Waterman}.

\medskip
\begin{figure}[ht!]
\def\hsp{\hspace{3.25pt}}
\drawmatching{0.33}{33}{4/12,5/11,6/10,13/21,14/20,15/19,22/30,23/29,24/28}
\tikz{\node at (0,0) {\small \sc g\hsp g\hsp g\hsp u\hsp a\hsp u\hsp n\hsp n\hsp n\hsp a\hsp u\hsp a\hsp g\hsp g\hsp g\hsp n\hsp n\hsp n\hsp c\hsp c\hsp c\hsp a\hsp u\hsp a\hsp n\hsp n\hsp n\hsp u\hsp a\hsp u\hsp c\hsp c\hsp c};
}

\vspace{-1ex}
\caption{Best first order secondary structure for {\bf s}, cf.\ \cite[Fig.~5.1]{Waterman}.}
\label{fig:best1stOrder}
\end{figure}

\begin{figure}[ht!]
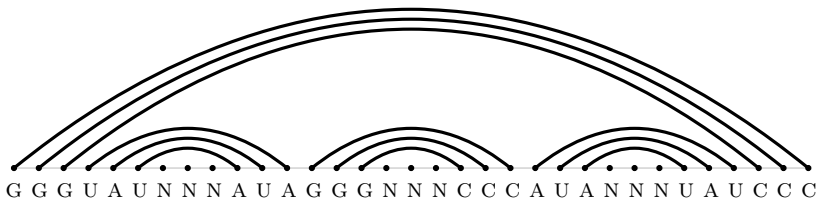

\def\hsp{\hspace{3.25pt}}
\drawmatching{0.33}{33}{1/33,2/32,3/31,4/12,5/11,6/10,13/21,14/20,15/19,22/30,23/29,24/28}
\tikz{\node at (0,0) {\small \sc g\hsp g\hsp g\hsp u\hsp a\hsp u\hsp n\hsp n\hsp n\hsp a\hsp u\hsp a\hsp g\hsp g\hsp g\hsp n\hsp n\hsp n\hsp c\hsp c\hsp c\hsp a\hsp u\hsp a\hsp n\hsp n\hsp n\hsp u\hsp a\hsp u\hsp c\hsp c\hsp c};
}

\vspace{-1ex}
\caption{Best second order secondary structure for {\bf s}, cf.\ \cite[Fig.~6.1]{Waterman}.}
\label{fig:best2ndOrder}
\end{figure}

In order to describe the arc diagrams we are interested in, we first recall some definitions. Let $\sigma$ be a permutation on $[n]=\{1,\dots,n\}$. The {\em reverse complement} of $\sigma$ is the permutation $\sigma^{rc}$ such that $\sigma^{rc}(i)=n+1-\sigma(n+1-i)$ for $i\in[n]$. Moreover, $\sigma$ is said to contain the pattern $3412$ if there are positions $i_1<i_2<i_3<i_4$ such that $\sigma(i_3)<\sigma(i_4)<\sigma(i_1)<\sigma(i_2)$. $\sigma$ is a $3412$-avoiding permutation if it does not contain a $3412$ pattern.

\medskip
We will consider four families of arc diagrams corresponding to:
\begin{enumerate}[(a)]
\itemsep3pt
\item 3412-avoiding involutions on $[n]$ with no transpositions of the form $(i, i+1)$. They correspond to arc diagrams with $n$ nodes such that there is no arc between adjacent nodes, every node belongs to at most one arc, and arcs do not cross. (Section~\ref{sec:noncrossing})
\item Involutions on $[n]$ with no transpositions of the form $(i, i+1)$. That is, arc diagrams like those of type (a), but allowing crossings. 
(Section~\ref{sec:crossings})
\item Noncrossing partitions of $[n]$ into blocks with no consecutive integers. (Section~\ref{sec:Motzkin})
\item Partitions of $[n]$ into blocks with no consecutive integers. (Section~\ref{sec:Bell})
\end{enumerate}
The arc diagrams of type (c) correspond to a smaller class of secondary structures with base triples than the ones considered in \cite{MuNe15}. For example, diagrams of the form\drawmatching{0.45}{4}{1/3,1/4}\, or\drawmatching{0.45}{4}{1/4,2/4}\, are not included in (c). However, (d) provides an extended model that includes pseudoknots. An example is given in Figure~\ref{fig:baseTripleEx}.

\begin{figure}[ht]
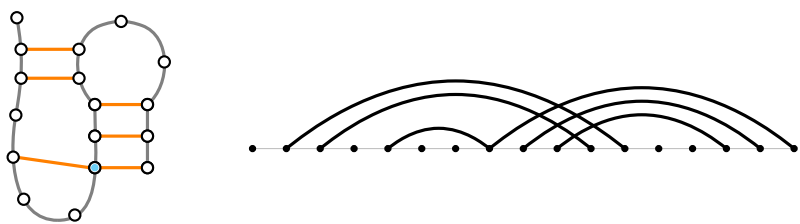

\tikzstyle{n}=[circle, inner sep=1.5, fill=white, draw=black, thick]
\tikzstyle{curve}=[arc1,gray]
\tikz[scale=0.7,baseline=6pt]{
 \draw[curve] (-1.48,2.85) to[out=-80,in=80] (-1.5,1) to[out=260,in=180] (-0.7,-1);
 \draw[curve] (-0.7,-1) to[out=0,in=270] (0,1.2);
 \draw[curve] (0,1.2) to[out=140,in=200] (0.1,2.7) to[out=20,in=160](0.9,2.7) to[out=-20,in=40](1,1.2);
 \draw[curve] (1,1.2) to[out=270,in=90] (1,0);
 \foreach \y in {3.4,4.5} {
 \draw[arc1,orange] (-1.4,0.5*\y) -- (-0.3,0.5*\y);
 	\node[n] at (-1.4,0.5*\y) {};
	\node[n] at (-0.3,0.5*\y) {};} 
 \draw[arc1,orange] (-1.5,0.2) -- (-0.1,0); 
 \foreach \y in {0,1,2} {\node[n] at (1,0.6*\y) {};}
 \foreach \y in {0,1,2} {\node[n] at (0,0.6*\y) {};}
 \foreach \y in {0,1,2} {\draw[arc1,orange] (0,0.6*\y) -- (1,0.6*\y);} 
 \foreach \y in {0,1,2} {\node[n] at (1,0.6*\y) {};}
 \foreach \y in {0,1,2} {\node[n] at (0,0.6*\y) {};}
 \draw[fill,cyan!50] (0,0) circle(0.06);
 \node[n] at (-1.48,2.85) {};
 \foreach \x/\y in {-1.5/1,-1.55/0.2,-1.35/-0.6,-0.38/-0.9} {\node[n] at (\x,\y) {};}
 \foreach \x/\y in {0.5/2.78,1.32/2.01} {\node[n] at (\x,\y) {};}
}
\hspace{3ex}
\drawmatching{0.45}{17}{2/12,3/11,5/8,8/17,9/16,10/15}
\caption{RNA structure with pseudoknots and a base triple.}
\label{fig:baseTripleEx}
\end{figure}

An arc diagram of type (a) or (b) is said to be {\em symmetric} if the corresponding involution is invariant under the reverse complement map. On the other hand, an arc diagram of type (c) or (d) is called {\em symmetric} if the corresponding partition is self-complementary. The {\em complement} of a partition $\pi$ of $[n]$ is the partition $\pi^c$ obtained by replacing $j$ by $n+1-j$ within each block of $\pi$. For example, if $\pi = 126\,|\,3\,|\,4\,|\,57$, then $\pi^c = 13\,|\,267\,|\,4\,|\,5$.

For each one of the above sets (a)--(d), we enumerate the subsets of symmetric elements by giving recurrence relations similar to the ones known for the entire class. In Section~\ref{sec:Motzkin}, we also consider the set $\L_{2n-1}$ of symmetric arc diagrams of type (c) with $2n-1$ nodes and give a bijection between $\L_{2n-1}$ and the set of base 3 (ternary) $n$-digit numbers with digit sum equal to $n$.

Finally, in Section~\ref{sec:summary}, we summarize our results and briefly discuss the asymptotic behavior of the sequences enumerating the various symmetric elements studied here.

\section{Symmetric noncrossing arc diagrams}
\label{sec:noncrossing}

For $n\in\mathbb{N}$ let $S_n$ denote the number of simple RNA secondary structures of a single-stranded nucleic acid with $n$ nucleotides. The sequence $(S_n)_{n\in\mathbb{N}}$ starts with $1$, $1$, $2$, $4$, $8$, $17$, $37$, $82$, $185$, $423$,\,\dots, and it is listed as A004148 in \cite{Sloane}. Such structures may be described using arc diagrams with non-intersecting arcs, where each node is connected to at most one arc and such that there is no arc connecting two adjacent nodes. For example, for $n=5$ there are 8 such arc diagrams, see Figure~\ref{fig:single5}.

\begin{figure}[ht]
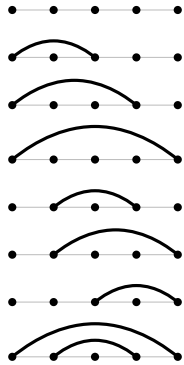

\def\ff{0.55}
\drawmatching{\ff}{5}{} \\[5pt]
\drawmatching{\ff}{5}{1/3} \\[5pt]
\drawmatching{\ff}{5}{1/4} \\[5pt]
\drawmatching{\ff}{5}{1/5} \\[5pt]
\drawmatching{\ff}{5}{2/4} \\[5pt]
\drawmatching{\ff}{5}{2/5} \\[5pt]
\drawmatching{\ff}{5}{3/5} \\[5pt]
\drawmatching{\ff}{5}{1/5,2/4}
\caption{Noncrossing arc diagrams with 5 nodes}
\label{fig:single5}
\end{figure}

\begin{remark}
If we label the nodes 1 through $n$ and think of elements connected by an arc as transpositions, then the above set of arc diagrams correspond to 3412-avoiding involutions on $[n]=\{1,\dots,n\}$ with no transpositions of the form $(i, i+1)$. For example, for $n=5$ the arc diagrams in Figure~\ref{fig:single5} correspond (from top to bottom) to the involutions $12345$, $32145$, $42315$, $52341$, $14325$, $15342$, $12543$, $54321$.
\end{remark}

\begin{theorem}[{\cite[Thm.~2.1]{Waterman}}] \label{thm:Waterman}
The sequence $(S_n)_{n\in\mathbb{N}}$ satisfies the recurrence relation:
\begin{equation*}
 S_n = S_{n-1} + \sum_{j=3}^n S_{j-2} S_{n-j} \;\text{ for } n\ge 3, 
\end{equation*}
where $S_0=1$, $S_1=1$, and $S_2=1$.
\end{theorem}
This is a consequence of the fact that every diagram of length $n$ either starts with an isolated node followed by a diagram of length $n-1$, or it starts with an arc $A_{1j}$ connecting the first node with a node at some position $j\ge 3$. In the latter case, one can have a diagram of length $j-2$ nested by the arc $A_{1j}$ and a diagram of length $n-j$ outside of $A_{1j}$.

In this paper, we are primarily interested in symmetric arc diagrams. More precisely, we say that an arc diagram describing a simple RNA secondary structure is {\em symmetric} if the corresponding involution is invariant under the reverse complement map.

For example, for $n=5$ the symmetric arc diagrams are the ones corresponding to the involutions $12345$, $1{4}3{2}5$, ${5}234{1}$, ${54}3{21}$. For $n=1,\dots,5$, the symmetric structures are:

\medskip
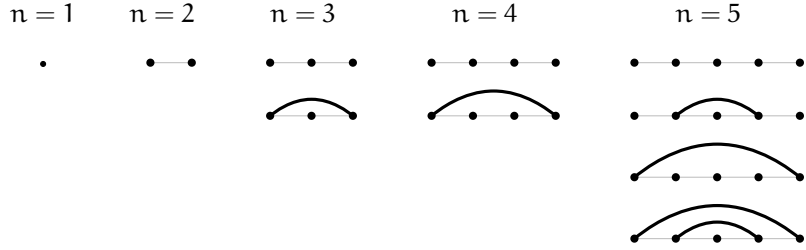
\begin{center}
\def\ff{0.55}
\small
\begin{tabular}{c@{\hskip20pt}c@{\hskip20pt}c@{\hskip20pt}c@{\hskip20pt}c}
$n=1$ & $n=2$ & $n=3$ & $n=4$ & $n=5$ \\[5pt] 
\tikz[scale=\ff]{\draw[fill] (0,0) circle (0.06);} & \drawmatching{\ff}{2}{} & \drawmatching{\ff}{3}{} & \drawmatching{\ff}{4}{} & \drawmatching{\ff}{5}{} \\[5pt]
&& \drawmatching{\ff}{3}{1/3} & \drawmatching{\ff}{4}{1/4} & \drawmatching{\ff}{5}{2/4} \\[5pt]
&&&& \drawmatching{\ff}{5}{1/5} \\[5pt]
&&&& \drawmatching{\ff}{5}{1/5,2/4}
\end{tabular}
\end{center}
\medskip

For $n\in\mathbb{N}$ let $R_n$ denote the number of simple symmetric RNA secondary structures with $n$ nucleotides. The sequence $(R_n)_{n\in\mathbb{N}}$ starts with $1, 1, 2, 2, 4, 5, 9, 12, 21, 29, 50, 71,\dots$, and it is listed as A088518 in \cite{Sloane}.

\begin{theorem} \label{thm:Rn}
The sequence $(R_n)_{n\in\mathbb{N}_0}$ satisfies the relation:
\begin{equation} \label{eqn:Rn}
 R_n = 2R_{n-2} + \sum_{j=3}^{\lfloor n/2\rfloor} S_{j-2} R_{n-2j} \;\text{ for } n\ge 6, 
\end{equation}
where $R_0=R_1=R_2=1$, $R_3=R_4=2$, and $R_5=4$.
\end{theorem}
\begin{proof}
For $n\ge 6$, there are three types of symmetric diagrams with $n$ nodes:
\begin{enumerate}[$(i)$]
\item Diagrams whose first and last nodes are isolated:
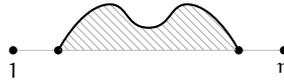
\begin{center}
\tikz[scale=0.6]{%
  \node[below=1pt] at (0,0) {\scriptsize $1$};
  \node[below=1pt] at (6,0) {\scriptsize $n$};
  \draw[gray!50] (0,0)--(6,0);
  \foreach \i in {0,1,5,6}{
    \draw[fill] (\i,0) circle (0.08);
  }
  \draw[patt1] plot [smooth, tension=1] coordinates {(1,0) (2,1) (3,0.5) (4,1) (5,0)};
}
\end{center}
\item Diagrams whose first and last nodes are connected by an arc:
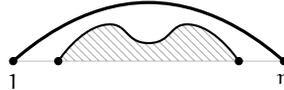
\begin{center}
\tikz[scale=0.6]{%
  \node[below=1pt] at (0,0) {\scriptsize $1$};
  \node[below=1pt] at (6,0) {\scriptsize $n$};
  \draw[gray!50] (0,0)--(6,0);
  \foreach \i in {0,1,5,6}{
    \draw[fill] (\i,0) circle (0.08);
  }
  \draw[arc1] (0,0) parabola bend (3,1.3) (6,0);
  \draw[patt1] plot [smooth, tension=1] coordinates {(1,0) (2,0.8) (3,0.4) (4,0.8) (5,0)};
}
\end{center}
\item Diagrams that start with an arc $A_{1j}$ connecting the first node with a node at some position $3\le j\le n/2$:
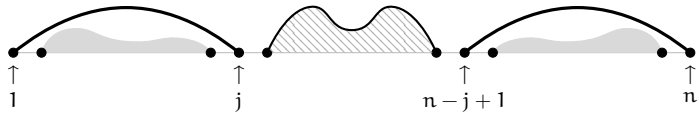
\begin{center}
\tikz[scale=0.75]{%
  \node[below=1pt] at (0,0) {\scriptsize $\uparrow$};
  \node[below=11pt] at (0,0) {\scriptsize $1$};
  \node[below=1pt] at (4,0) {\scriptsize $\uparrow$};
  \node[below=11pt] at (4,0) {\scriptsize $j$};
  \node[below=1pt] at (8,0) {\scriptsize $\uparrow$};
  \node[below=11pt] at (8,0) {\scriptsize $n-j+1$};
  \node[below=1pt] at (12,0) {\scriptsize $\uparrow$};
  \node[below=11pt] at (12,0) {\scriptsize $n$};
  \draw[gray!50] (0,0)--(12,0);
  \foreach \i in {0,4,4.5,7.5,8,12}{
    \draw[fill] (\i,0) circle (0.08);
  }
  \draw[arc1] (0,0) parabola bend (2,0.8) (4,0);
  \draw[arc1] (8,0) parabola bend (10,0.8) (12,0);
  \draw[patt1] plot [smooth, tension=1] coordinates {(4.5,0) (5.2,0.8) (6,0.4) (6.8,0.8) (7.5,0)};
  \draw[fill,gray!30] plot [smooth, tension=1] coordinates {(0.5,0) (1,0.42) (2,0.2) (3,0.25) (3.5,0)};
  \draw[fill,gray!30] plot [smooth, tension=1] coordinates {(8.5,0) (9,0.25) (10,0.2) (11,0.42) (11.5,0)};
  \foreach \i in {0.5,3.5,8.5,11.5}{
    \draw[fill] (\i,0) circle (0.08);
  }
 }
\end{center}
\end{enumerate}

There are $R_{n-2}$ diagrams of type $(i)$ and $R_{n-2}$ diagrams of type $(ii)$. Finally, the sum over $j$ on the right-hand side of \eqref{eqn:Rn} comes from the fact that we can place any diagram (not necessarily symmetric) of length $j-2$ underneath $A_{1j}$ (and the mirror image underneath the arc $A_{n-j+1,n}$), and at positions $j+1$ through $n-j$, one can have any symmetric arc diagram of length $n-2j$.
\end{proof}

For diagrams of odd and even length, the OEIS lists two recurrences for which we provide combinatorial proofs.
\begin{proposition}
  For $n\in\mathbb{N}$, we have
  \begin{gather*}
    R_{2n+1} = R_{2n} + R_{2n-1} \,\text{ and }\; R_{2n} = R_{2n-1} + R_{2n-2} - S_{n-1}.
  \end{gather*}
\end{proposition}
\begin{proof}
The first recurrence follows from the fact that every symmetric arc diagram with an odd number of nodes has an isolated middle node. Given any arc diagram $D_{2n+1}$ with $2n+1$ nodes, there are two disjoint cases: either the node in the middle lies underneath a short arc of length 3 (i.e.\!\drawmatching{0.6}{3}{1/3}), or it does not. In the first case, one can get a diagram with $2n-1$ nodes by removing from $D_{2n+1}$ the arc around the middle node. In the second case, one can get a diagram with $2n$ nodes by removing from $D_{2n+1}$ just the middle node. 

Conversely, every symmetric diagram with $2n$ nodes can be extended by adding one single node between the two nodes in the middle, and every symmetric diagram with $2n-1$ nodes can be extended by adding a short arc around the middle node. Clearly, these two procedures give different arc diagrams with $2n+1$ nodes. 

In order to prove the recurrence relation for $R_{2n}$, let $D_{2n}$ be a symmetric arc diagram with $2n$ nodes. For the central nodes at positions $n$ and $n+1$, there are two disjoint cases:
\begin{itemize}
\item The two nodes are isolated nodes.
\item There are arcs $A_{j,n}$ and $A_{n+1,2n-j+1}$ for some $j\le n-2$.
\end{itemize}
Clearly, every $D_{2n}$ of the first type uniquely corresponds to a diagram with $2n-1$ nodes, obtained from $D_{2n}$ by merging the two central nodes. There are $R_{2n-1}$ such diagrams.

On the other hand, if $D_{2n}$ contains arcs $A_{j,n}$ and $A_{n+1,2n-j+1}$ for some $j$, we can map it to a diagram $D_{2n-2}$ with $2n-2$ nodes as follows: replace $A_{j,n}$ and $A_{n+1,2n-j+1}$ by a single arc connecting the nodes $j$ and $2n-j+1$, and then remove the two central nodes of $D_{2n}$. 

Note that the only symmetric arc diagrams with $2n-2$ nodes that are not in the image of the above map are diagrams of the form:
\smallskip
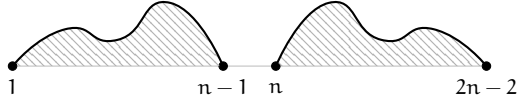
\begin{center}
\tikz[scale=0.7]{%
  \node[below=1pt] at (1,0) {\scriptsize $1$};
  \node[below=1pt] at (5,0) {\scriptsize $n-1$};
  \node[below=2.8pt] at (6,0) {\scriptsize $n$};
  \node[below=1pt] at (10,0) {\scriptsize $2n-2$};
  \draw[gray!50] (1,0)--(10,0);
  \foreach \i in {1,5,6,10}{
    \draw[fill] (\i,0) circle (0.08);
  }
  \draw[patt1] plot [smooth, tension=1] coordinates {(1,0) (2,0.7) (3,0.5) (4,1.2) (5,0)};
  \draw[patt1] plot [smooth, tension=1] coordinates {(6,0) (7,1.2) (8,0.5) (9,0.7) (10,0)};
}
\end{center}
Since there are $S_{n-1}$ such diagrams, we conclude that the number of symmetric arc diagrams $D_{2n}$ such that $D_{2n}$ contains arcs $A_{j,n}$ and $A_{n+1,2n-j+1}$ for some $j$ is $R_{2n-2} - S_{n-1}$.

In conclusion, $R_{2n} = R_{2n-1} + R_{2n-2} - S_{n-1}$, as stated.
\end{proof}

In \cite{Sloane}, Emeric Deutsch pointed out that $(R_n)_{n\in\mathbb{N}}$ also satisfies the relation:
\begin{equation*}
 R_n = F_{n} - \sum_{j=1}^{\lfloor n/2\rfloor-1} S_{j} F_{n-1-2j} \;\text{ for } n\ge 4, 
\end{equation*}
where $R_0=R_1=R_2=1$,  $R_3=2$, and $(F_n)_{n\in\mathbb{N}}$ denotes the Fibonacci sequence. This can be easily shown using the above proposition.

\section{Symmetric arc diagrams with crossings}
\label{sec:crossings}

We now turn to a larger class of arc diagrams that include crossings. 

For $n\in\mathbb{N}$ ,we let $\P_n$ be the set of arc diagrams with $n$ nodes where each node is connected to at most one arc, and such that there is no arc connecting two adjacent nodes. These arc diagrams correspond to involutions with no transpositions of the form $(i, i+1)$. Observe that a crossing of two arcs corresponds to a 3412 pattern in the involution.

There are no crossings for diagrams with less than 4 nodes, and\drawmatching{0.6}{4}{1/3,2/4} \ is the only 4-node arc diagram with crossings. If $P_n=|\P_n|$, the sequence $(P_n)_{n\in\mathbb{N}}$ starts with $1$, $1$, $2$, $5$, $13$, $37$, $112$, $363$, $1235$,\,\dots, and it is listed as A170941 in \cite{Sloane}. 

\begin{theorem}[{\cite[Thm.~5]{HaSt99}}] 
The sequence $(P_n)_{n\in\mathbb{N}}$ satisfies the recurrence relation:
\[ P_n = P_{n-1} + (n-1)P_{n-2}-P_{n-3}+P_{n-4} \;\text{ for } n\ge 5, \]
where $P_1=P_2=1$, $P_3=2$, and $P_4=5$.
\end{theorem}
A simple combinatorial proof of this theorem in terms of arc diagrams can be found in \cite[Section~3]{GanSte11}. For more information and other interesting formulas, see \cite{HaSt99,JQR08}.

For $0\le k\le n$, we let $P(n,k)$ be the number of diagrams in $\P_n$ having exactly $k$ free (isolated) nodes. A few of the terms generated by $P(n,k)$ are listed in Table~\ref{tab:PseudoK}. Clearly, we have $P_n = \sum_{k=0}^n P(n,k)$, and it is easy to check that
\[ P(1,0)=P(2,0)=P(2,1)=0 \;\text{ and }\; P(1,1)=P(2,2)=1. \]
\begin{theorem} \label{thm:pseudo_nk}
For $n\ge 3$ we have
\[ P(n,k) = (k+1)P(n-1,k+1) + P(n-1,k-1) - P(n-2,k), \]
where $P(n,k)=0$ if $k<0$ or $k>n$.
\end{theorem}
\begin{proof}
For $n\ge 3$ and $0\le k\le n$, the set of diagrams counted by $P(n,k)$ can be split into three disjoint sets (possibly empty):
\begin{enumerate}[\;$(i)$]
\item Diagrams whose first node is isolated.
\item Diagrams whose first node is part of an arc and whose second node is free.
\item Diagrams where neither one of the first two nodes is isolated.
\end{enumerate}
For example, for $n=5$ the above split gives the following 3 groups:

\medskip
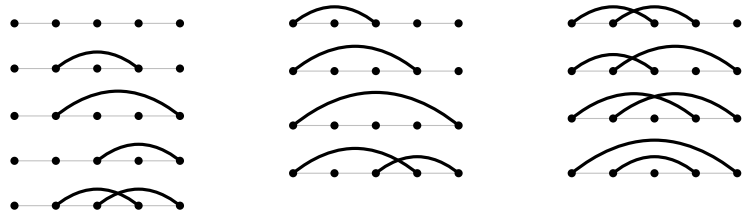
\begin{center}
\def\ff{0.55}
\begin{tabular}{ccc}  
\parbox[t]{.22\textwidth}{
\drawmatching{\ff}{5}{} \\[4pt]
\drawmatching{\ff}{5}{2/4} \\[5pt]
\drawmatching{\ff}{5}{2/5} \\[4pt]
\drawmatching{\ff}{5}{3/5} \\[4pt]
\drawmatching{\ff}{5}{2/4,3/5}
}
&
\parbox[t]{.22\textwidth}{
\drawmatching{\ff}{5}{1/3} \\[5pt]
\drawmatching{\ff}{5}{1/4} \\[5pt]
\drawmatching{\ff}{5}{1/5} \\[5pt]
\drawmatching{\ff}{5}{1/4,3/5}
}
&
\parbox[t]{.22\textwidth}{
\drawmatching{\ff}{5}{1/3,2/4} \\[5pt]
\drawmatching{\ff}{5}{1/3,2/5} \\[5pt] 
\drawmatching{\ff}{5}{1/4,2/5} \\[5pt]
\drawmatching{\ff}{5}{1/5,2/4}
}
\end{tabular}
\end{center}
\medskip
There are $P(n-1,k-1)$ diagrams of the first type. Moreover, every diagram of type $(ii)$ can be constructed from a diagram $D$ with $n-2$ nodes, having $k$ of them free, by adding two nodes to the left of $D$ and then drawing an arc from the left-most added node to any of the $k$ free nodes of $D$. This process gives $kP(n-2,k)$ diagrams of the second type.

Finally, there are $P(n-1,k+1)-P(n-2,k)$ arc diagrams with $n-1$ nodes, having $k+1$ of them free, and such that the left-most node is not free. Every diagram of type $(iii)$ can be constructed from one of these by adding a node to the left and drawing an arc from that node to one of the existing $k+1$ free nodes. In other words, there are $(k+1)\left(P(n-1,k+1)-P(n-2,k)\right)$ diagrams of the third type.

The statement follows by adding the three quantities.
\end{proof}

\begin{table}[ht]
\small
\def\R{\rule[-1ex]{0ex}{3.5ex}}
\begin{tabular}{c|@{\hskip8pt}rrrrrrrr|r}
  $n\backslash k$ & 0 & 1 & 2 & 3 & 4 & 5 & 6 & 7 & $\Sigma$ \\[1pt] \hline
  \R 1&0&1&&&&&&&1 \\
  \R 2&0&0&1&&&&&&1 \\
  \R 3&0&1&0&1&&&&&2 \\
  \R 4&1&0&3&0&1&&&&5 \\
  \R 5&0&6&0&6&0&1&&&13 \\
  \R 6&5&0&21&0&10&0&1&&37 \\
  \R 7&0&41&0&55&0&15&0&1&112
\end{tabular}
\vspace{2ex}
\caption{Numbers generated by $P(n,k)$.}
\label{tab:PseudoK}
\end{table}

We now address the enumeration of symmetric arc diagrams in $\P_n$. Recall that such a diagram is symmetric if the associated involution is invariant under the reverse complement map. For example, for $n=5$ there are 5 such symmetric elements:
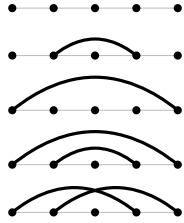
\begin{center}
\def\ff{0.55}
\drawmatching{\ff}{5}{} \\[5pt]
\drawmatching{\ff}{5}{2/4} \\[5pt]
\drawmatching{\ff}{5}{1/5} \\[5pt]
\drawmatching{\ff}{5}{1/5,2/4} \\[5pt]
\drawmatching{\ff}{5}{1/4,2/5} 
\end{center}

\medskip
For $n\in\mathbb{N}$ and $0\le k\le n$, let $Q(n,k)$ be the number of symmetric arc diagrams with $n$ nodes, $k$ of which are isolated, and such that each of the remaining $n-k$ nodes is connected by an arc to exactly one nonadjacent node. In particular, $n$ and $k$ must have the same parity. The total number of symmetric arc diagrams in $\P_n$ is then given by
\[ Q_n = \!\!\sum\limits_{\stackrel{0\le k\le n}{k\equiv n\hspace{-6pt}\pmod 2}} \!\!Q(n,k). \]
The sequence $(Q_n)_{n\in\mathbb{N}}$ starts with $1, 1, 2, 3, 5, 11, 16, 43, 59, 179, 238, 801,\dots$.

The middle node of a symmetric diagram with $2n+1$ nodes is always isolated and its adjacent nodes are either connected to each other by an arc or they are not. Each diagram of the latter type corresponds to a diagram with $2n$ nodes, obtained by removing the middle node, and each diagram of the first type corresponds to a diagram with $2n-1$ nodes obtained by removing the short arc around the middle node. Therefore,
\begin{equation} \label{eq:Q_odd}
 Q_{2n+1} = Q_{2n} + Q_{2n-1}.
\end{equation}

We now focus on diagrams with an even number of nodes. One can easily check that
\[ Q(2,0) = 0, \; Q(2,2) = 1, \text{ and } Q(4,0) = Q(4,2) = Q(4,4) = 1. \]
\begin{theorem}
For $n\ge 3$ and $k$ even, we have
\begin{equation*}
 Q(2n,k) = (k+2)Q(2n-2,k+2) + Q(2n-2,k) + Q(2n-2,k-2) - Q(2n-4,k),
\end{equation*}
where $Q(2n,k)=0$ if $k<0$ or $k>2n$
\end{theorem}
\begin{proof}
The above identity follows from a simple combinatorial argument similar to the one we used for the proof of Theorem~\ref{thm:pseudo_nk}.

For $n\ge 3$ and $0\le k\le 2n$, the set of diagrams counted by $Q(2n,k)$ can be split into four disjoint sets (possibly empty) depending on the nature of the first and second nodes:
\begin{enumerate}[\;$(i)$]
\item The first node is isolated.
\item The first node is connected to the last node by a single arc.
\item The first node is connected to a node different from the last node, and the second node is isolated.
\item The first node is connected to a node different from the last, and the second node is not isolated.
\end{enumerate}
There are $Q(2n-2,k-2)$ arc diagrams of the first type and $Q(2n-2,k)$ of the second type. Now, every diagram of the third type can be constructed as follows: Take a symmetric diagram $D$ with $2n-4$ nodes, having $k$ of them free, and extend it with 4 nodes (2 on each side) to create a new diagram $D'$ with $2n$ nodes. The first node of $D'$ can then be connected to any of the $k$ free nodes of $D$, or to the second to last node of $D'$. As a last step, connect the last node of $D'$ to a free node in order to create a symmetric diagram $D''$. This construction generates $(k+1)Q(2n-4,k)$ diagrams of the third type.

Finally, there are $(k+2)[Q(2n-2,k+2)-Q(2n-4,k)]$ diagrams of the fourth type. This can be seen with a similar argument as before. Just observe that every diagram of the fourth type may be constructed from one of the $Q(2n-2,k+2)-Q(2n-4,k)$ arc diagrams with $2n-2$ nodes, having $k+2$ of them free, and such that the first node is not isolated. The claimed recurrence then follows by combining the four disjoint sets.
\end{proof}

\begin{table}[ht!]
\small
\def\R{\rule[-1ex]{0ex}{3.5ex}}
\begin{tabular}{c|@{\hskip8pt}rrrrrrrr|r}
  $n\backslash k$ & 0 & 2 & 4 & 6 & 8 & 10 & 12 & 14 & $\Sigma$ \\[1pt] \hline
  \R 1&0&1&&&&&&&1 \\
  \R 2&1&1&1&&&&&&3 \\
  \R 3&3&5&2&1&&&&&11 \\
  \R 4&12&15&12&3&1&&&&43 \\
  \R 5&39&70&43&22&4&1&&&179 \\
  \R 6&167&266&233&94&35&5&1&&801 \\
  \R 7&660&1295&1020&585&175&51&6&1&3793
\end{tabular}
\vspace{2ex}
\caption{Numbers generated by $Q(2n,k)$.}
\label{tab:SymPseudoK}
\end{table}

In particular, we have $Q(2n,2n-2) = n-1$ for $n\ge 1$, and the sequence $(a_n)_{n\in\mathbb{N}}$ defined by $a_n=Q(2(n+1),2(n-1))$ satisfies the recurrence relation
\[ a_1=1, \quad a_n = a_{n-1} + 3n-2 \;\text{ for } n\ge 2, \]
thus it is the sequence of pentagonal numbers (see \cite[A000326]{Sloane}).

We finish this section by pointing out that, based on computational data up to $n=1000$, the sequence $(q_n)_{n\in\mathbb{N}}$ defined by $q_n = Q_{2n}$ seems to satisfy the  recurrence relation:
\[ q_n = 3q_{n-1}+2(n-2)q_{n-2}-2(n-2)q_{n-3}+3q_{n-4}-q_{n-5} \;\text{ for } n\ge 6, \]
where $q_1=1$, $q_2=3$, $q_3=11$, $q_4=43$, and $q_5=179$.

\section{Symmetric arc diagrams of Motzkin type}
\label{sec:Motzkin}

In this section, we focus on noncrossing arc diagrams that have no arc connecting two adjacent nodes, and where nonnesting arcs are allowed to intersect at a node. These arc diagrams correspond to noncrossing partitions of $[n]$ into blocks with no consecutive integers, and they are counted by the Motzkin numbers. For example, for $n=5$ we have the 8 diagrams shown in Figure~\ref{fig:single5} together with the diagram:

\begin{center}  
  \drawmatching{0.7}{5}{1/3,3/5}
\end{center}

Let $\M_n$ denote the set of such diagrams with $n$ nodes and let $M_n=|\M_n|$. With a similar argument as the one used to prove Theorem~\ref{thm:Waterman}, one can show that the sequence $(M_n)_{n\in\mathbb{N}}$ satisfies the recurrence relation:
\begin{equation*} 
 M_n = M_{n-1} + \sum_{j=3}^{n} M_{j-2} M_{n-j+1} \text{ for } n\ge 3, 
\end{equation*}
where $M_1=M_2=1$. As mentioned above, $(M_n)_{n\in\mathbb{N}}$ is the sequence of Motzkin numbers $1, 1, 2, 4, 9, 21, 51, 127, 323, 835, \dots$, see \cite[A001006]{Sloane}.

For this reason, we call this class of diagrams {\em Motzkin arc diagrams}. As expected, there is a natural bijection $\varphi$ between the set of Motzkin paths\footnote{Lattice paths with steps $\U=(1,1)$, $\D=(1,-1)$, and $\H=(1,0)$, staying weakly above the $x$-axis.} from $(0,0)$ to $(n-1,0)$ and the set $\M_n$. The map $\varphi$ works as follows: Given a Motzkin path with $n-1$ steps, project its vertices down to the baseline of the diagram, and for each matching pair of up and down steps on the path, draw an arc on the diagram from the node corresponding to the beginning of the up-step to the node corresponding to the end of the matching down-step. For example:

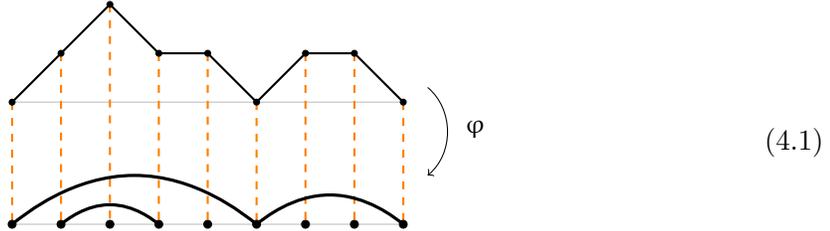
\begin{equation}\label{eqn:phi}
\begin{tikzpicture}[scale=0.65,baseline=-18pt]
\draw[thin,gray!50] (1,0)--(9,0);
\def\dis{2.5}
\foreach \x/\y in {1/0,2/1,3/2,4/1,5/1,6/0,7/1,8/1,9/0}{
	\draw[thick,dashed,orange] (\x,\y) -- (\x,-\dis);
	\draw[fill] (\x,\y) circle (0.06);
}
\draw[thick] (1,0)--(3,2)--(4,1)--(5,1)--(6,0)--(7,1)--(8,1)--(9,0);
\draw[gray!50] (1,-\dis)--(9,-\dis);
\foreach \i in {1,...,9} {%
	\draw[fill] (\i,-\dis) circle (0.08);
}
\foreach \x/\y in {1/6,2/4,6/9} {
	\pgfmathsetmacro\s{\x+\y}
	\pgfmathsetmacro\d{\y-\x}
	\draw[arc1] (\x,-\dis) parabola bend (0.5*\s,0.2*\d-\dis) (\y,-\dis);
}
\draw[->] (9.5,0.3) to[out=-40, in=40] (9.5,-1.5);
\node[right=1pt] at (10,-0.5) {\small $\varphi$};
\end{tikzpicture}
\end{equation}

\bigskip\noindent
This map is invertible and preserves symmetry, cf.\ \cite[Exercise~6.38]{Sta99}. 

Let $\L_n$ be the set of symmetric Motzkin arc diagrams with $n$ nodes, and let $L_n=|\L_n|$. 
\begin{theorem}
The sequence $(L_n)_{n\in\mathbb{N}}$ satisfies the relation
\begin{equation} \label{eqn:Ln}
 L_n = 2L_{n-2} + \sum_{j=3}^{\lfloor \frac{n+1}{2}\rfloor} M_{j-2} L_{n+2-2j} \text{ for } n>4, 
\end{equation}
where $L_0=L_1=L_2=1$, and $L_3=L_4=2$.
\end{theorem}
\begin{proof}
The proof follows almost the same argument as the proof of Theorem~\ref{thm:Rn}, except that the third type of diagrams described there should be replaced by arc diagrams of the form:
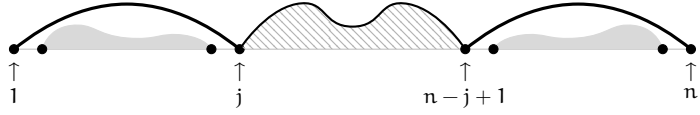
\begin{center}
\tikz[scale=0.75]{%
  \node[below=1pt] at (0,0) {\scriptsize $\uparrow$};
  \node[below=11pt] at (0,0) {\scriptsize $1$};
  \node[below=1pt] at (4,0) {\scriptsize $\uparrow$};
  \node[below=11pt] at (4,0) {\scriptsize $j$};
  \node[below=1pt] at (8,0) {\scriptsize $\uparrow$};
  \node[below=11pt] at (8,0) {\scriptsize $n-j+1$};
  \node[below=1pt] at (12,0) {\scriptsize $\uparrow$};
  \node[below=11pt] at (12,0) {\scriptsize $n$};
  \draw[thin,gray!50] (0,0)--(12,0);
  \foreach \i in {0,4,8,12} {%
	\draw[fill] (\i,0) circle (0.08);
  }  
  \draw[arc1] (0,0) parabola bend (2,0.8) (4,0);
  \draw[arc1] (8,0) parabola bend (10,0.8) (12,0);
  \draw[patt1] plot [smooth, tension=1] coordinates {(4,0) (5,0.8) (6,0.4) (7,0.8) (8,0)};
  \draw[fill,gray!30] plot [smooth, tension=1] coordinates {(0.5,0) (1,0.42) (2,0.2) (3,0.25) (3.5,0)};
  \draw[fill,gray!30] plot [smooth, tension=1] coordinates {(8.5,0) (9,0.25) (10,0.2) (11,0.42) (11.5,0)};
  \foreach \i in {0.5,3.5,8.5,11.5} {%
	\draw[fill] (\i,0) circle (0.08);
  }  
 }
\end{center}

There are $L_{n-2}$ Motzkin diagrams of type $(i)$ and $L_{n-2}$ Motzkin diagrams of type $(ii)$ (see proof of Theorem~\ref{thm:Rn}). The sum over $j$ on the right-hand side of \eqref{eqn:Ln} comes from the fact that we can place any Motzkin diagram of length $j-2$ underneath $A_{1j}$, and at positions $j$ through $n-j+1$, one can have any symmetric Motzkin diagram of length $n+2-2j$.
\end{proof}

The first few terms of $(L_n)_{n\in\mathbb{N}}$ are given by 
\[ 1, 1, 2, 2, 5, 5, 13, 13, 35, 35, 96, 96, 267, 267, 750, 750,\dots \]
This suggests a bijection between $\L_{2n}$ and $\L_{2n-1}$. Indeed, any symmetric Motzkin diagram $D_{2n}$ with $2n$ nodes can be mapped into a symmetric Motzkin diagram with $2n-1$ nodes by merging the two central nodes of $D_{2n}$. This map is clearly bijective.

Moreover, the sequence $1, 2, 5, 13, 35, 96, 267, 750,\dots$ (say $(L_{2n-1})_{n\in\mathbb{N}}$) appears to be the same as \cite[A005773]{Sloane}. This is indeed the case:

\begin{theorem}
Let $(a_n)_{n\in\mathbb{N}_0}$ be defined by $a_0=1$ and $a_n=L_{2n-1}$ for $n\ge 1$. Then
\begin{equation} \label{eqn:oddLn}
  a_n = \sum_{j=0}^{n-1} M_{j+1} a_{n-1-j},
\end{equation}
where $M_{j+1}$ is the number of Motzkin paths with $j$ steps. Moreover, for $n\ge 1$, we have
\[ a_{n+1} = 3a_{n} - M_n. \]
\end{theorem}
\begin{proof}
First observe that by means of \eqref{eqn:phi}, the set $\L_{2n-1}$ is in one-to-one correspondence with the set of symmetric Motzkin paths with $2n-2$ steps. We split the latter set into $n$ disjoint sets according to the length of the longest initial Motzkin factor:

\begin{itemize}
\item[$\circ$] There are $a_{n-1}$ paths of the form $\U\PM\D$, where $\PM$ is a symmetric Motzkin path.
\item[$\circ$] There are $M_n$ paths of the form $\PM_{n}\PM_{n}'$, where $\PM_{n}$ is a Motzkin path from $(0,0)$ to $(n-1,0)$, and $\PM_{n}'$ is its mirror image going from $(n-1,0)$ to $(2n-2,0)$.
\item[$\circ$] The remaining paths are of the form $\PM_{j+1}\U \PM \D\PM_{j+1}'$, where $\PM_{j+1}$ is a Motzkin path from $(0,0)$ to $(j,0)$ for some $1\le j\le n-2$, $\PM$ is a symmetric Motzkin path from $(j+1,1)$ to $(2n-3-j,1)$, and $\PM_{j+1}'$ is the mirror image of $\PM_{j+1}$ going from $(2n-2-j,0)$ to $(2n-2,0)$. There are $M_{j+1} a_{n-1-j}$ such paths since $\PM$ has $2(n-1-j) - 2$ steps.
\end{itemize}

Formula \eqref{eqn:oddLn} then follows by adding all the terms over $j$ from 0 to $n-1$.

Finally, the simpler relation $a_{n+1} = 3a_{n} - M_n$ is a direct consequence of \eqref{eqn:Ln} and \eqref{eqn:oddLn}. Indeed, replacing $n$ by $2n+1$ in \eqref{eqn:Ln}, we get
\begin{align*}
 a_{n+1} = L_{2n+1} &= 2L_{2n-1} + \sum_{j=3}^{n+1} M_{j-2} L_{2n+3-2j} \\
 &= 2a_{n} + \sum_{j=3}^{n+1} M_{j-2} a_{n+2-j} = 2a_{n} + \sum_{j=0}^{n-2} M_{j+1} a_{n-1-j}.
\end{align*}
Together with \eqref{eqn:oddLn} this implies $a_{n+1} = 2a_{n} + (a_n - M_n) = 3a_{n} - M_n$.
\end{proof}

We finish this section with a bijective proof of a statement by John Layman who claimed in \cite[A005773]{Sloane} that the above sequence $(a_n)_{n\in\mathbb{N}}$ also enumerates the set of base 3 (ternary) $n$-digit numbers with digit sum equal to $n$. We denote this set by $\T_n$.

Note that, since symmetric Motzkin paths with $2n-2$ steps are uniquely determined by the first $n-1$ steps, the map $\varphi$ from \eqref{eqn:phi} induces a bijection $\varphi_\ell$ from $\L_{2n-1}$ to the set of left factors of length $n-1$ of Motzkin paths. Furthermore, if we use the numbers 2, 1, and 0 to denote the steps $(1,1)$, $(1,0)$, and $(1,-1)$, respectively, then every left factor of a Motzkin path may be represented by a number in base 3. We denote the latter map by $\tau$.

\begin{figure}[ht]
\def\ff{0.55}
\begin{tabular}{c@{\hskip25pt}c@{\hskip15pt}c@{\hskip20pt}c@{\hskip20pt}c}
\drawmatching{\ff}{5}{} & \tikz[baseline=-1]{\draw[dashed,->] (0,0)--(0.8,0); \node[above=1pt] at (0.4,0) {\small $\varphi_\ell$}} & \drawpath{0.5}{2}{0}{0,0,0} &  \tikz[baseline=-1]{\draw[dashed,->] (0,0)--(0.8,0); \node[above=1pt] at (0.4,0) {\small $\tau$}} & {\small\tt 11} \\[5pt]
\drawmatching{\ff}{5}{2/4} & \tikz[baseline=-5]{\draw[dashed,->] (0,0)--(0.8,0)} & \drawpath{0.45}{2}{1}{0,0,1} &  \tikz[baseline=-5]{\draw[dashed,->] (0,0)--(0.8,0)} & {\small\tt 12} \\[5pt]
\drawmatching{\ff}{5}{1/5} & \tikz[baseline=-5]{\draw[dashed,->] (0,0)--(0.8,0)} & \drawpath{0.45}{2}{1}{0,1,1} &  \tikz[baseline=-5]{\draw[dashed,->] (0,0)--(0.8,0)} & {\small\tt 21} \\[5pt]
\drawmatching{\ff}{5}{1/3,3/5} & \tikz[baseline=-5]{\draw[dashed,->] (0,0)--(0.8,0)} & \drawpath{0.45}{2}{1}{0,1,0} &  \tikz[baseline=-5]{\draw[dashed,->] (0,0)--(0.8,0)} & {\small\tt 20} \\[5pt]
\drawmatching{\ff}{5}{1/5,2/4} & \tikz[baseline=-5]{\draw[dashed,->] (0,0)--(0.8,0)} & \drawpath{0.45}{2}{2}{0,1,2} &  \tikz[baseline=-5]{\draw[dashed,->] (0,0)--(0.8,0)} & {\small\tt 22}
\end{tabular}
\vspace{1ex}
\caption{Elements of $\L_5$ and their images under $\varphi_\ell$ and $\tau\circ \varphi_\ell$.}
\label{fig:case5}
\end{figure}

\begin{theorem} 
The sets $\L_{2n-1}$ and $\T_n$ are equinumerous. 
\end{theorem}
\begin{proof}
Let $D\in \L_{2n-1}$. We start by considering some special cases.

If $\varphi_\ell(D)$ is a Motzkin path of length $n-1$, then we define
\[  \psi(D) = \tau(\varphi_\ell(D))+\mathtt{1}, \]
and if $\varphi_\ell(D)$ is a path that ends at the point $(n-1,1)$, we let
\[  \psi(D) = \tau(\varphi_\ell(D))+\mathtt{0}. \]
Here ``+ {\tt a}'' means to append {\tt a} at the end of the word $\tau(\varphi_\ell(D))$. Note that $\psi(D)\in \T_n$.

If $n>2$ and $D\in\L_{2n-1}$ consists of $n-1$ nested arcs and a single isolated node in the middle, then $\varphi_\ell(D)$ is the straight path from $(0,0)$ to $(n-1,n-1)$. In this case, we let
\begin{equation}\label{eqn:extreme}
  \psi(D) = 
  \begin{cases}
    {\tt 20}^{k}{\tt 2}^{k-1} &\text{if } n=2k, \\
    {\tt 10}^{k}{\tt 2}^{k} &\text{if } n=2k+1,
  \end{cases}
\end{equation}
where $a^k$ denotes the word made out of $k$ consecutive copies of $a$. By definition, $\psi(D)$ belongs to $\T_n$.
For example, 
\[ \drawmatching{0.3}{9}{1/9,2/8,3/7,4/6} \quad \longleftrightarrow\quad \text{{\tt 10022}}\in\T_5.  \]

If $D\in\L_{2n-1}$ is an arc diagram such that the path $\varphi_\ell(D)$ ends at the point $(n-1,m)$ with $1<m<n-1$, then $\varphi_\ell(D)$ admits a unique factorization of the form
\[ \varphi_\ell(D) = \PM_0 \U \PM_1 \U^{k_1} \PM_2 \U^{k_2} \cdots \PM_{i-1} \U^{k_i} \PM_{i}, \]
where $\U^{j}$ denotes an ascent of length $j$, $k_1+\dots+k_i=m-1$, and each $\PM_j$ is a Motzkin path (possibly the empty path $\varepsilon$). Note that $D_m = \varphi_\ell^{-1}(\U^m)$ is the diagram in $\L_{2m+1}$ consisting of $m$ nested arcs and a single isolated node in the middle. Therefore, $\psi(D_m)$ is of the form \eqref{eqn:extreme} with $n$ replaced by $m+1$. Factor $\psi(D_m)$ as
\[ \psi(D_m) =  \gamma_0 \gamma_{k_1} \gamma_{k_2} \cdots \gamma_{k_i},  \]
where $\mathop{length}( \gamma_{k_j})=k_j$ for $j=2,\dots,i$, and $\gamma_0=\begin{cases} \mathtt{2} &\text{for $m$ odd}\\ \mathtt{10} &\text{for $m$ even}\end{cases}$.

\smallskip
Finally, with $\tau(\varepsilon)$ defined as the empty word, we let
\begin{equation*}
  \psi(D) = \tau(\PM_0) \gamma_0 \tau(\PM_1) \gamma_{k_1} \cdots \tau(\PM_{i-1}) \gamma_{k_i} \tau(\PM_{i}).
\end{equation*}
Once again, our construction gives $\psi(D)\in\T_n$.
\end{proof}

An example that illustrates our map $\psi$ is given in Figure~\ref{fig:mapPhi}.

\begin{figure}[ht!]
\begin{tikzpicture}[scale=0.75]
\node at (0,2.8) {\drawmatching{0.45}{19}{1/4,4/16,6/14,7/9,11/13,16/19}};
\draw[->] (0,1.7) -- (0,1.2);
\node at (0,0) {\drawpath{0.45}{9}{3}{0,1,1,0,1,1,2,3,2,2}};
\draw[->] (0,-1.2) -- (0,-1.7);
\node at (0,-2.1) {{\tt 2101012201}.};
\end{tikzpicture}
\caption{Example of map from $\L_{2n-1}$ to $\T_n$.}
\label{fig:mapPhi}
\end{figure}
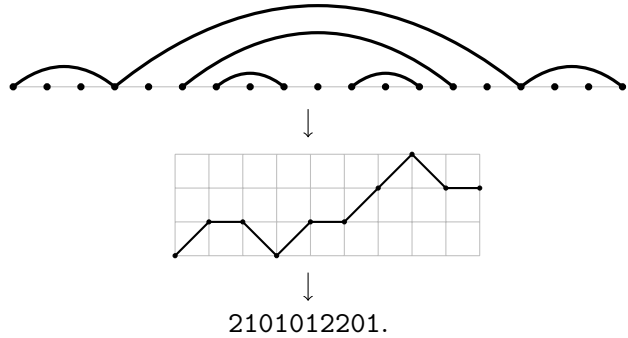

The elements of $\T_3$ corresponding to the arc diagrams in Figure~\ref{fig:case5} (top to bottom) are
\[ \mathtt{111},\quad \mathtt{210},\quad \mathtt{120},\quad \mathtt{102},\quad \mathtt{201}. \]

\section{Symmetric arc diagrams of Bell type}
\label{sec:Bell}

While the sets of diagrams considered in Section~\ref{sec:Motzkin} are of interest from a combinatorial point of view, they are not good enough to model RNA secondary structures with pseudoknots and base triples like the one in Figure~\ref{fig:baseTripleEx}. In this section, we will study the set $\B_n$ of arc diagrams with $n$ nodes that have the following properties:
\begin{itemize}
\item[$\circ$] there are no arcs connecting two adjacent nodes,
\item[$\circ$] arcs may cross,
\item[$\circ$] nonnesting arcs may intersect at a node.
\end{itemize}
In other words, the elements of $\B_n$ correspond to partitions of $[n]$ into blocks with no consecutive integers. Thus $|\B_n|$ is the Bell number $B_{n-1}$, cf.\ \cite[A000110]{Sloane}.

For example, for $n=5$ there are 15 such diagrams (see Figure~\ref{fig:bell5}).

\begin{figure}[ht]
\def\ff{0.55}
\drawmatching{\ff}{5}{} \qquad \drawmatching{\ff}{5}{2/5} \qquad \drawmatching{\ff}{5}{1/4,2/5}\\[5pt]
\drawmatching{\ff}{5}{1/3} \qquad \drawmatching{\ff}{5}{3/5} \qquad \drawmatching{\ff}{5}{1/4,3/5}\\[5pt]
\drawmatching{\ff}{5}{1/4} \qquad \drawmatching{\ff}{5}{1/5,2/4} \qquad \drawmatching{\ff}{5}{2/4,3/5}\\[5pt]
\drawmatching{\ff}{5}{1/5} \qquad \drawmatching{\ff}{5}{1/3,2/4} \qquad \drawmatching{\ff}{5}{1/3,3/5}\\[5pt]
\drawmatching{\ff}{5}{2/4} \qquad \drawmatching{\ff}{5}{1/3,2/5} \qquad \drawmatching{\ff}{5}{1/3,3/5,2/4}
\caption{Elements of $\B_5$.}
\label{fig:bell5}
\end{figure}

Each element of $\B_n$ may be decomposed into blocks that we call {\em arc-connected} components. A single node is arc-connected, a single arc is arc-connected, and every subdiagram consisting of consecutive arcs joined by nodes is arc-connected. For example, the diagram

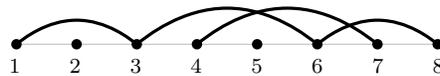
\begin{center}
\tikz[scale=0.7]{%
  \node at (5,0) {\drawmatching{0.8}{8}{1/3,3/6,6/8,4/7}};
  \foreach \i in {1,...,8}{
	\node[below=8pt] at (1.15*\i,0) {\scriptsize \i};
  }
}
\end{center}
has 4 arc-connected components that correspond to the blocks of the partition $1368\,|\,2\,|\,47\,|\,5$.

Let $\A_n\subset \B_n$ be the subset of symmetric elements, and let $A(n,k)$ be the number of elements in $\A_n$ having $k+1$ arc-connected components. Clearly, $A(1,0)=1$, $A(n,0)=0$ for $n>1$, and $A(n,n-1)=1$. Moreover, for $n>1$ it can be easily checked that
\[ A(n,1) =1 \;\text{ and }\; A(n,n-2) = \lfloor\tfrac{n-1}{2}\rfloor. \]
In particular, $A(2,1)=1$, $A(3,1)=A(3,2)=1$, and $A(4,1)=A(4,2)=A(4,3)=1$.
\begin{theorem}
For $n\ge 5$ and $2\le k\le n-3$, we have
\[ A(n,k) = kA(n-2,k) + A(n-2,k-1) + A(n-2,k-2). \]
\end{theorem}
\begin{proof}
For $n\ge 5$, the elements of $\A_{n}$ with $k+1$ arc-connected components can be split into three disjoint sets of diagrams depending on the nature of the first and last nodes:
\begin{enumerate}[\;$(i)$]
\item The first and last nodes are isolated.
\item The first and last nodes are connected by a single arc.
\item The first node is connected to a node other than the last node.
\end{enumerate}
In the first case, the other $n-2$ nodes must be part of a symmetric arc diagram with $k-1$ arc-connected components. There are $A(n-2,k-2)$ such diagrams. In the second case, the first and last nodes form an arc-connected component, so the other $n-2$ nodes must be part of a symmetric diagram with $k$ components, and there are $A(n-2,k-1)$ of those.

Finally, if the first node is connected to a node other than the last node, it must be part of one of the $k$ arc-connected components that do not contain the second node. There are $kA(n-2,k)$ such arc diagrams and the claimed relation follows by combining all 3 cases.
\end{proof}

Table~\ref{tab:symStirling} illustrates the triangle generated by $A(n,k)$. These numbers are related to the enumeration of achiral color patterns, see \cite[A304972]{Sloane}.
\begin{table}[ht]
\small
\def\R{\rule[-1ex]{0ex}{3.5ex}}
\begin{tabular}{c|rrrrrrrrrr|r}
  $n\backslash k$ & 0 & 1 & 2 & 3 & 4 & 5 & 6 & 7 & 8 & 9 & $\Sigma$ \\[1pt] \hline
  \R 1&1&&&&&&&&&&1 \\
  \R 2&0&1&&&&&&&&&1 \\
  \R 3&0&1&1&&&&&&&&2 \\
  \R 4&0&1&1&1&&&&&&&3 \\
  \R 5&0&1&3&2&1&&&&&&7 \\
  \R 6&0&1&3&5&2&1&&&&&12 \\
  \R 7&0&1&7&10&9&3&1&&&&31 \\
  \R 8&0&1&7&19&16&12&3&1&&&59 \\
  \R 9&0&1&15&38&53&34&18&4&1&&164 \\
  \R 10&0&1&15&65&90&95&46&22&4&1&339
\end{tabular}
\vspace{2ex}
\caption{Triangle generated by $A(n,k)$.}
\label{tab:symStirling}
\end{table}

The total number of elements in $\A_n$ is then given by
\[ A_n = \sum_{k=0}^{n-1} A(n,k). \]
The sequence $(A_n)_{n\in\mathbb{N}}$ starts with $1, 1, 2, 3, 7, 12, 31, 59, 164, 339,\dots$, see \cite[A080107]{Sloane}.

\section{Final remarks}
\label{sec:summary}

The primary goal of this paper was to enumerate the symmetric elements in each of the four families of arc diagrams described in the introduction. For the sequences $(R_n)_{n\in\mathbb{N}}$, $(Q_n)_{n\in\mathbb{N}}$, $(L_n)_{n\in\mathbb{N}}$, and $(A_n)_{n\in\mathbb{N}}$, we found recurrence relations and connections to other combinatorial objects. The beginning terms of these sequences, next to the corresponding sequence that counts all elements of the same type, are listed in Table~\ref{tab:allSequences}. As mentioned in Section~\ref{sec:Motzkin}, the sequence $(L_{2n-1})_{n\in\mathbb{N}}$ is listed as A005773 in \cite{Sloane}.

\begin{table}[ht!]
\small
\def\R{\rule[-1ex]{0ex}{3.7ex}}
\begin{tabular}{c|l|c}
\R  & Sequence & OEIS \cite{Sloane} \\\hline\hline
\R $R_n$ & $1, 1, 2, 2, 4, 5, 9, 12, 21, 29, 50, 71,\dots$ & A088518\\
 $S_n$ & $1, 1, 2, 4, 8, 17, 37, 82, 185, 423, 978, 2283, \dots$ & A004148\\[2pt] \hline\hline
\R $Q_n$ & $1, 1, 2, 3, 5, 11, 16, 43, 59, 179, 238, 801,\dots$ & N/A \\
 $P_n$ & $1, 1, 2, 5, 13, 37, 112, 363, 1235, 4427, 16526, 64351,\dots$ & A170941\\[2pt] \hline\hline
\R $L_n$ & $1, 1, 2, 2, 5, 5, 13, 13, 35, 35, 96, 96,\dots$ & N/A \\
 $M_n$ & $1, 1, 2, 4, 9, 21, 51, 127, 323, 835, 2188, 5798,\dots$ & A001006\\[2pt] \hline\hline
\R $A_n$ & $1, 1, 2, 3, 7, 12, 31, 59, 164, 339, 999, 2210,\dots$ & A080107\\
 $B_n$ & $1, 1, 2, 5, 15, 52, 203, 877, 4140, 21147, 115975, 678570,\dots$ & A000110
\end{tabular}
\vspace{2ex}
\caption{Sequences discussed in sections \ref{sec:noncrossing} through \ref{sec:Bell}.}
\label{tab:allSequences}
\end{table}

For simple RNA secondary structures, corresponding to involutions with no transpositions of the form $(i,i+1)$, there are known\footnote{See \cite[Eq.~(26)]{StWat79} and \cite[A088518]{Sloane}.} asymptotic estimates:
\[ S_n \sim c_1\cdot \frac{(1+\varphi)^n}{n^{3/2}}  \;\text{ and }\; R_n \sim c_2\cdot\frac{\varphi^n}{n^{1/2}} \;\text{ as } n\to\infty, \]
where $c_1\approx 1.104$, $c_2\approx 0.863$, and $\varphi=\frac{1+\sqrt{5}}{2}$. Thus,
\[ \frac{R_n}{S_n} \sim \frac{c_2}{c_1}\cdot \frac{n}{\varphi^n}\; \text{ as } n\to\infty. \]

Moreover, for the arc diagrams of Motzkin type, we have\footnote{See \cite[A001006]{Sloane} and \cite[A005773]{Sloane}.}
\[ M_n \sim \sqrt{\frac{3}{4\pi}}\cdot \frac{3^{n}}{n^{3/2}} \;\text{ and }\; L_n < \sqrt{\frac{2}{\pi}}\cdot\frac{3^{n/2}}{n^{1/2}} \;\text{ as } n\to\infty, \]
which implies
\[ \frac{L_n}{M_n} < \sqrt{\frac{8}{3}} \cdot \frac{n}{3^{n/2}} \;\text{ as } n\to\infty. \] 

For the other two families, numerical data seems to indicate that the ratio of the number of symmetric elements to the total number of elements of the same type, converges even faster to 0, see Table~\ref{tab:numData}. 

\begin{table}[ht]
\def\R{\rule[-1ex]{0ex}{3.5ex}}
\small
\begin{tabular}{c|cccc}
$n$ & $R_n/S_n$ & $L_n/M_n$ & $Q_n/P_n$ & $A_n/B_n$ \\[2pt] \hline\hline
\R 5 & 0.2941176471 & 0.2380952381 & 0.2972972973 & 0.2307692308 \\
6   & 0.2432432432 & 0.2549019608 & 0.1428571429 & 0.1527093596 \\
7   & 0.1463414634 & 0.1023622047 & 0.1184573003 & 0.0672748005 \\
8   & 0.1135135135 & 0.1083591331 & 0.0477732794 & 0.0396135266 \\
9   & 0.0685579196 & 0.0419161677 & 0.0404337023 & 0.0160306426 \\
10 & 0.0511247444 & 0.0438756856 & 0.0144015491 & 0.0086139254 \\
11 & 0.0310994306 & 0.0165574336 & 0.0124473590 & 0.0032568490 \\
12 & 0.0225200074 & 0.0172135904 & 0.0040043011 & 0.0016235535 \\
13 & 0.0137416569 & 0.0063822158 & 0.0034992873 & 0.0005799720 \\
14 & 0.0097458185 & 0.0066001373 & 0.0010349767 & 0.0002712948 \\
15 & 0.0059589192 & 0.0024148990 & 0.0009145496 & 0.0000922971 \\
16 & 0.0041594968 & 0.0024875010 & 0.0002514919 & 0.0000408516 \\
17 & 0.0025473928 & 0.0009008057 & 0.0002238335 & 0.0000133153 \\
18 & 0.0017558071 & 0.0009249765 & 0.0000577643 & 0.0000056122 \\
19 & 0.0010765606 & 0.0003322109 & 0.0000517725 & 0.0000017607 \\
20 & 0.0007344722 & 0.0003402618 & 0.0000126234 & 0.0000007103 \\
\end{tabular}
\vspace{2ex}
\caption{}
\label{tab:numData}
\end{table}


\end{document}